\documentclass[11pt,a4paper]{article}
\date{}
\usepackage{caption2}
\usepackage{CJK}
\usepackage{tabls}
\usepackage{bm}
\usepackage{multirow}
\usepackage{amssymb}
\usepackage{amsfonts}
\usepackage{mathrsfs}
\usepackage{empheq}
\usepackage{amsmath}
\usepackage{amsthm}
\usepackage{graphicx}
\usepackage{color}
\usepackage{indentfirst}
\usepackage{hyperref}
\usepackage{float}
\usepackage{cases}
\usepackage{setspace}
\allowdisplaybreaks%

 \numberwithin{equation}{section}
\theoremstyle{definition}

\usepackage{graphicx}%
\usepackage{multirow}%
\usepackage{amsmath,amssymb,amsfonts}%
\usepackage{amsthm}%
\usepackage{mathrsfs}%
\usepackage[title]{appendix}%
\usepackage{xcolor}%
\usepackage{textcomp}%
\usepackage{manyfoot}%
\usepackage{booktabs}%
\usepackage{algorithm}%
\usepackage{algorithmicx}%
\usepackage{algpseudocode}%
\usepackage{listings}%
\usepackage{braket}
\usepackage{appendix, authblk}

\theoremstyle{definition}
\newtheorem{Def}{Deffinition}[section]
\newtheorem{Prop}[Def]{Proposition}
\newtheorem{lemma}[Def]{Lemma}
\newtheorem{theorem}[Def]{Theorem}
\newtheorem{remark}[Def]{Remark}

\newcommand{\N}{\mathbb{N}} 
\newcommand{\R}{\mathbb{R}} 

\newcommand{\al}{\alpha}
\newcommand{\bt}{\beta}
\newcommand{\bR}{\bar{R}}
\newcommand{\bS}{\bar{S}}

\makeatletter

\@addtoreset{equation}{section}
\newcommand{\bV}{\bar{V}}
\newcommand{\bW}{\bar{W}}
\newcommand{\wa}{w_{\alpha}}
\newcommand{\wz}{w_{\zeta}}
\newcommand{\wb}{w_{\beta} }
\newcommand{\wg}{w_{\gamma} }

\newcommand{\og}{\Omega_T}
 \newcommand{\bog}{\overline{\Omega_T}}
 \newcommand{\wze}{w_{\zeta,\varepsilon} }
 \newcommand{\wbe}{w_{\beta,\varepsilon} }

\begin{document}

\title{Loss of regularity  for solutions to  1D degenerate quasilinear wave equations}


\author[1]{Yanbo Hu%
}
\author[2,*]{Yuusuke Sugiyama%
}

\affil[1]{Department of Mathematics, Zhejiang University of Science and Technology, Hangzhou 310023, Zhejiang, PR China\\
\texttt{yanbo.hu@hotmail.com}}
\affil[2]{Tokyo University of Science, 1-3 Kagurazaka, Shinjuku-ku, Tokyo 162-8601, Japan\\
\texttt{sugiyama.y@rs.tus.ac.jp}}

\affil[*]{Corresponding author: \texttt{sugiyama.y@rs.tus.ac.jp}}

\maketitle
\begin{abstract}
In this paper, we study the loss of regularity for solutions to
one-dimensional degenerate wave equations. We first consider the linear
equation
\(
u_{tt}=\bigl(c(t,x)^2u_x\bigr)_x
\)
with \(c(t,x)\sim x^p\), and then the quasilinear equation
\(
u_{tt}=(u^{2a}u_x)_x.
\)
The initial data are assumed to satisfy
\(
u_0(x)\sim x^\alpha
\)
and
\(
u_1(x)\sim x^\beta
\)
near the degenerate point \(x=0\). In the linear problem, \(p\)
describes the strength of the degeneracy, while \(\alpha\) and
\(\beta\) describe the regularity of the initial data. In the
quasilinear problem, \(\alpha\) also determines the initial degeneracy
through \(u_0(x)^a\sim x^{a\alpha}\).

Our main concern is the actual occurrence of loss of regularity, namely,
the phenomenon in which the regularity of \(u(t,\cdot)\) near \(x=0\)
becomes lower than that of \(u_0\) for \(t>0\). Previously, such a loss
had mainly been established for special linear equations with
coefficients depending only on time. In our previous work on the
quasilinear equation, we proved local well-posedness without loss of
regularity when \(\beta\geq\alpha\). In the present paper, we show that
this condition is also necessary. More precisely, if
\(\beta<\alpha\), then
\(
C_1tx^\beta\leq u(t,x)\leq C_2x^\beta
\)
near \(x=0\) for sufficiently small positive time. Thus the order of
the solution changes from that of \(u_0\) to that of \(u_1\), and an
actual loss of regularity occurs for both linear equations with
time-space dependent coefficients and degenerate quasilinear equations.
\end{abstract}


\medskip


\tableofcontents


\section{Introduction}
\subsection{Degeneracy and  loss of regularity }
In this paper, we study the loss of regularity for one-dimensional
degenerate wave equations. We first consider the following linear
degenerate wave equation with a variable coefficient:
\begin{align} \label{vdwv}
   u_{tt} = \left(c(t,x)^2 u_x\right)_x .
\end{align}
Here the coefficient \(c(t,x)\) is allowed to vanish at the degenerate
point \(x=0\). This linear equation is used to clarify the basic
mechanism of the loss of regularity.

We then consider the following one-dimensional degenerate quasilinear
wave equation:
\begin{align} \label{eq}
   u_{tt} =  ( u^{2a}  u_x )_x
\end{align}
with initial conditions
\begin{align} \label{ini}
  u(0,x) = u_0(x) \quad \mbox{and} \quad u_t(0,x) = u_1(x).
\end{align}
The equation \eqref{eq} describes the shearing motion in elastoplastic
materials (see \cite{NC}). If \(u_0\geq c_0>0\) for some positive
constant \(c_0\), then the equation is non-degenerate near \(t=0\), and
\eqref{eq} can be solved locally in time with a suitable boundary
condition, such as a non-zero Dirichlet, Neumann, or periodic boundary
condition, or on the whole space \(\R\).

In the present paper, we consider the case where the equation
degenerates. In the linear equation \eqref{vdwv}, this corresponds to
the vanishing of \(c(t,x)\) at \(x=0\), while in the quasilinear equation
\eqref{eq}, it corresponds to the vanishing of \(u\) at the degenerate
point. From a physical point of view, the degeneracy describes an abrupt
change in the properties of a material, such as yielding. Despite this
physical motivation, the existence and regularity theory for such
degenerate equations is not complete, mainly because the principal part
of the equation degenerates.

In particular, we deal with the loss of regularity as a phenomenon
characteristic of degenerate equations. Our aim is to investigate
conditions under which such a loss of regularity actually occurs;
combined with our previous work \cite{YHYS1}, this gives a necessary
and sufficient characterization.

We illustrate the notion of regularity loss through a simple example of a linear wave equation with variable coefficients:
\begin{align}
  u_{tt} - a(t,x)^2 u_{xx} = 0, \quad t>0,\ x \in \mathbb{R},
\end{align}
with smooth initial data
\begin{align}
  u(0,x) = u_0(x), \quad u_t(0,x) = u_1(x).
\end{align}
If there exists \(c_0>0\) such that \(a(t,x)\geq c_0\), and if
\(a(t,x)\) is smooth, then the solution  preserves the classical and Sobolev regularities of the initial data for $m \in \N$ and $s \in \R$ :
\[
(u_0,u_1) \in C^m \times C^{m-1}
\quad\Longrightarrow\quad
(u(t,\cdot),u_t(t,\cdot) ) \in C^m \times C^{m-1}
\]
and
\[
(u_0,u_1) \in H^{s} \times H^{s-1}
\quad\Longrightarrow\quad
(u(t,\cdot),u_t(t,\cdot) )  \in H^{s} \times H^{s-1}.
\]

Namely, no regularity loss occurs.
However, for  variable coefficients $a(t,x)$ that degenerate at some point, the characteristic speed vanishes at the degenerate point.  
In this case, the following estimate is known (e.g. Oleinik \cite{ol}, Han \cite{qh}, Han, Hong and Lin \cite{HanHongLin2006} and  Ascanelli and  Cicognani \cite{AscanelliCicognani2005}):
\begin{align}
  \|u(t,\cdot)\|_{H^{s}} \le C \bigl( \|u_0\|_{H^{s+\sigma}} + \|u_1\|_{H^{s-1+\sigma}} \bigr),
\end{align}
where $\sigma>0$ depends on the order of degeneracy.
Based on such an estimate, the phenomenon in which the regularity order of the solution is lower than that of the initial data is called regularity loss.
However, even for linear degenerate equations, the occurrence of such phenomena is known only in limited and rather exceptional situations (for instance, when $a(t,x)=t^l$ for some $l>0$). 
In particular, in the case where the propagation speed depends on the spatial variable $x$, very little has been known about the actual occurrence of regularity loss.
Naturally, for quasilinear problems, no results have been available.

In our previous work \cite{YHYS1}, we studied solutions whose initial
profiles behave like
\[
  u_0(x)\sim x^\alpha,\qquad u_1(x)\sim x^\beta
\]
near the degenerate point \(x=0\). We proved that, if
\(\beta\geq\alpha\), then the corresponding solution preserves the
order of \(u_0\) near \(x=0\) for a short time. More precisely, there
exist positive constants \(C_1\) and \(C_2\) such that
\[
  C_1 x^\alpha \leq u(t,x)\leq C_2 x^\alpha
\]
near \(x=0\). In this sense, no loss of regularity occurs in this
regime.

We recall that the exponent \(\alpha\) has two roles. It represents the
strength of the degeneracy of the equation, and at the same time it
determines the regularity of the function near the degenerate point.
For example, a function behaving like \(x^\alpha\) near \(x=0\) has
regularity determined by the exponent \(\alpha\). Thus the condition
\(\beta\geq\alpha\) means that \(u_1\) is no less regular than \(u_0\)
near the degenerate point.

The purpose of the present paper is to prove the converse phenomenon.
We show that the condition \(\beta\geq\alpha\) is sharp. More precisely,
if \(\beta<\alpha\), then every solution in the natural class loses the
regularity determined by \(u_0\). In fact, near \(x=0\), the leading
order of the solution changes from the order of \(u_0\) to that of
\(u_1\); for sufficiently small positive time, the solution satisfies
\[
  C_1 t x^\beta\leq u(t,x)\leq C_2 x^\beta .
\]

\subsection{Main theorems}

 Let \(\delta>0\) be fixed. We consider the linear equation \eqref{vdwv}
on the interval \(x\in[0,\delta]\). The coefficient
\(c\in C^1([0,T]\times[0,\delta])\) is assumed to satisfy, for
\((t,x)\in[0,T]\times[0,\delta]\),
\begin{align}
L_0x^p\leq c(t,x)\leq L_1x^p,\qquad
L_2x^{p-1}\leq c_x(t,x)\leq L_3x^{p-1},
\label{c-assumption}
\end{align}
and
\begin{align}
|c_t(t,x)|\leq L_4x^p,
\label{ct-assumption}
\end{align}
where \(L_0,\ldots,L_4\) are positive constants. Furthermore, the
initial data \((u_0,u_1)\) satisfy, for \(x\in[0,\delta]\),
\begin{align}\label{inid}
K_0x^\alpha\leq u_0(x)\leq K_1x^\alpha,\qquad
K_2x^\beta\leq u_1(x)\leq K_3x^\beta,
\end{align}
where \(K_0,\ldots,K_3\) are positive constants.

We establish the following result for this equation. The theorem shows
that, if \(\beta<\alpha\), then the leading order of the solution near
the degenerate point \(x=0\) is governed by \(u_1\), and a loss of
regularity occurs.

\begin{theorem}\label{thl}
Let \(p>1\) and \(\alpha>\beta>0\). Any solution to
\eqref{vdwv} with initial data \eqref{inid} that satisfies
\begin{align}
u\in C([0,T]\times[0,\delta])\cap C^2([0,T]\times(0,\delta]),
\label{27}\\
u_t,\ c(t,x)u_x\in C([0,T]\times[0,\delta])
\label{28}
\end{align}
also satisfies the following estimate, provided that \(T>0\) is
sufficiently small:
\[
C_1tx^\beta\leq u(t,x)\leq C_2x^\beta,
\]
for \((t,x)\in[0,T]\times[0,\delta_1]\), where
\(\delta_1<\delta\) is sufficiently small.
\end{theorem}

The following result can be obtained by applying the argument used to establish the loss of regularity for the linear equation to the quasilinear case.
\begin{theorem} \label{thq}
Let $a \geq1$ and $\alpha > \bt > 0$. 
Any solution of \eqref{eq} with  initial data \eqref{inid}, satisfying the following conditions:
\begin{align}
u \in C^2([0,T] \times(0,\delta]) \cap C([0,T] \times [0,\delta]), \label{29}  \\
u_t, u^{a-1} u_x \in  C([0,T] \times [0,\delta]),\label{29-2} \\
u(t,0)=u_t (t,0) =0, \label{30} \\
u(t,x) \geq 0, \label{31}
\end{align}
also satisfies the following estimate, provided that \(T>0\) is sufficiently small:
\[
C_1 t x^{\beta} \leq u(t,x) \leq C_2 x^{\beta},
\]
for \((t,x)\in[0,T]\times[0,\delta_1]\), where
\(\delta_1<\delta\) is sufficiently small.
\end{theorem}

\begin{remark}\label{re1}
  It is not always necessary to assume the nonnegativity of $u(t,x)$. For instance, if $a$ is a natural number and, in addition, $u^m u_{xx}$ is assumed to be bounded for some $m \in \N$ such that $m <2a$, then for sufficiently small $T$ one can deduce the non-negativity of $u(t,x)$.

\end{remark}
In the proof of Theorem \ref{thq}, the divergence-free structure of the equation is not utilized.  
In fact, a similar theorem also holds for the following variational nonlinear wave equation:
\begin{align}\label{veq}
u_{tt} = u^a (u^a u_x)_x.
\end{align}
\begin{theorem} \label{thq2}
Under the same assumptions as in Theorem~\ref{thq}, the solution to the initial value problem \eqref{veq}, \eqref{ini} satisfies the same conclusion as in Theorem~\ref{thq}.
\end{theorem}

\begin{remark}[Loss of differentiability]
The main estimates show that, for \(t>0\), the leading order of
\(u(t,x)\) near \(x=0\) is \(x^\beta\), rather than \(x^\alpha\).
Thus the regularity at the degenerate point is reduced from that
determined by \(u_0\) to that determined by \(u_1\).

This loss can be interpreted in the sense of Peano differentiability.
We recall that Peano derivatives can be characterized in terms of
higher order difference quotients; see Ash~\cite{Ash1970-Peano} and
Ash, Catoiu and Fejzi\'c~\cite{AshCatoiuFejzic2024-TwoPointwise}.
For a function whose leading behavior near \(x=0\) is \(x^\sigma\),
the differentiability at the degenerate point is determined by the
exponent \(\sigma\). Hence, if
\[
  \lfloor \beta \rfloor < \lfloor \alpha \rfloor,
\]
then the order of Peano differentiability at \(x=0\) drops after
positive time.

Moreover, if \(0<\beta<1\), then the loss is also a loss of classical
differentiability. Indeed, since \(u(t,0)=0\), the lower bound
\[
  u(t,x)\geq Ctx^\beta
\]
implies
\[
  \frac{u(t,x)-u(t,0)}{x}
  \geq Ct x^{\beta-1}\to\infty
  \qquad (x\downarrow0).
\]
Thus, in this case, the solution is not classically differentiable at
\(x=0\) for any fixed \(t>0\).
\end{remark}

\subsection{Known results}
The well-posedness of non-degenerate quasi-linear wave equations has been widely known.
Kato \cite{tk} and  Hughes, Kato and Marsden \cite{HKM} have shown an abstract theorem about the well-posedness
of the system of  general quasi-linear wave equations in the Sobolev space.
In the one-dimensional case, the well-posedness in $C^1 _b$ class for first order  hyperbolic equations has been studied  by Douglis \cite{D} and Hartman and Wintner \cite{HW2} (see also Majda \cite{m} and Courant and Lax \cite{CL}),
where $C^1 _b $ is a set of continuous and bounded functions whose derivatives are also bounded.
In order to apply these results to the existence problem for \eqref{eq}, the following assumption is required:
\begin{eqnarray} \label{non-deg}
 u_0 (x) \geq   c_0 >0
\end{eqnarray}
for a constant $c_0$. This condition ensures that the equation \eqref{eq} is of the strictly hyperbolic type near $t=0$. 
\eqref{eq} can be written formally as the following hyperbolic system:
\begin{align*}
\left\{
\begin{aligned}
v_t &= u^{2a} u_x,\\
u_t &= v_x ,
\end{aligned}
\right.
\end{align*}
where $v = \int u_t dx $.
That is, we set
\[
U =
\begin{pmatrix}
u \\
v
\end{pmatrix}, 
\quad
A(U) =
\begin{pmatrix}
0 & 1 \\
u^{2a} & 0
\end{pmatrix}.
\]
Then the system can be written as
\[
U_t = A(U)\, U_x.
\]
In other words, the non-degeneracy of the equation means that the eigenvalues of the matrix $A(U)$ are distinct.
In the following, we review previous studies on degenerate hyperbolic equations.
The existence, nonexistence and regularity of solutions to the following type of  linear weakly hyperbolic equations have been studied by many authors (e.g. Oleinik \cite{ol}, Colombini and  Spagnolo  \cite{CS}, Ivrii and Petkov \cite{IV} and Taniguchi and  Tozaki \cite{TT}),
\begin{eqnarray} \label{wl}
u_{tt} - \sum_{i,j=1} ^n  a_{i,j}(t,x)  \partial_{x_i} \partial_{x_j}  u + \sum_{j=1} ^n b_j (t,x)  \partial_{x_j}u=0,
\end{eqnarray}
where $a_{i,j}$ and $b_j$ are smooth functions and $\sum_{i,j=1} ^n  a_{i,j}(t,x) \xi_i \xi_j \geq 0$ is assumed for $(\xi_1 , \ldots , \xi_n) \in \mathbb{R}^n$. We note that  $\sum_{i,j=1} ^n a_{i,j}(t,x) \xi_i \xi_j =0$ corresponds to the degeneracy.
In Oleinik \cite{ol}, \eqref{wl} has been solved under the so-called Levi condition:
 \begin{eqnarray*} 
  C_1   \left(  \sum_{j=1}^n b_j \xi_j \right)^2 \leq C_2 \sum_{i,j=1} ^n \left( a_{i,j} \xi_i \xi_j + \partial_t a_{i,j} \xi_i \xi_j  \right) .
\end{eqnarray*}
Regardless of assuming the Levi condition, we can only obtain  the following energy estimate with the regularity loss for weakly hyperbolic equations:
\begin{eqnarray} \label{re-loss}
\| u \|_{H^{s}} + \| u_t \|_{H^{s-1}} \leq C(\| u_0 \|_{H^{s+r_1}} + \| u_1 \|_{H^{s-1 +r_2}}),
\end{eqnarray}
where $s$ is an arbitrary real number and $r_1$ and $r_2$ are  non-negative numbers. 
However, the loss of regularity is known only in a few examples.  
Qi \cite{Qi1958} has considered the problem:
\[
    v_{tt} - t^2 v_{xx} = b v_x, \quad 
    v(x,0) = \varphi(x), \quad 
    v_t(x,0) = 0, \quad x \in \mathbb{R}.
\]
If $b = 4m + 1$ with $m \in \mathbb{N}$, the solution has the form
\[
    v(x,t) = \sum_{j=0}^m C_j t^{2j} \partial_x^j \varphi \!\left(x + \tfrac{1}{2}t^2 \right),
\]
with constants $C_j$ and non-vanishing $C_m$.  
If $\varphi \in H^s$, then
\[
    v(\cdot,t) \in H^{s-m}, \quad (t>0).
\]
This example has been generalized by Taniguchi–Tozaki \cite{TT} and Yagdjian \cite{Yagdjian1997} to a broader class of variable-coefficient wave equations whose principal part depends on time.  
However, when the coefficients of the equation depend on the spatial variables, such results had not been known. Galstian and Kinoshita in \cite{GalstianKinoshita2016} provide representation formulas for solutions to the following one–dimensional linear degenerate wave equations:
\[
\partial_t^{2}u(t,x)\;-\;|x|^{2k}\,\partial_x^{2}u(t,x)\;=0,
\]
and
\[
\partial_t^{2}u(t,x)\;-\;t^{2\ell}x^{2k}\,\partial_x^{2}u(t,x)\;=\;0.
\]
Their formulas involve special functions, including the Gauss hypergeometric functions and Bessel functions. However, whether a loss of regularity occurs at \(x=0\) is not explicitly stated.
In Section 5, we use the explicit formula given in that paper 
to compute the actual loss of regularity for solutions of 
\(\partial_t^{2}v - x^{4}\partial_x^{2}v = 0\),
and we confirm that it agrees with our main result.

Colombini–Spagnolo \cite{CS} constructed a smooth oscillatory time-dependent coefficient $a(t)\ge0$ such that the Cauchy problem
\[
u_{tt}-a(t)u_{xx}=0
\]
is not well-posed in $C^\infty$.  
This explicit counterexample shows that even smooth weakly hyperbolic equations with oscillating coefficients may lose well-posedness, in contrast to the analytic data case where well-posedness is preserved.
When the principal part is non-degenerate, it is known that if the coefficients are not smooth 
(for example, merely Lipschitz or H{\"o}lder continuous), 
then an actual loss of regularity can occur, and the solvability may be adversely affected (see Cicognani and Colombini \cite{Cicognani2006} and \cite{ColombiniSpagnolo1989}).
From the viewpoint of regularity and solvability for nonlinear wave equations, the adverse effects of degeneracy on solutions are much fewer and more limited. First, we review results concerning the existence of solutions for nonlinear problems. Manfrin in \cite{rm0, rm1, rm2} and D'Ancona and Manfrin in \cite{DM} have established the local existence and the uniqueness for quasilinear wave equations  including \eqref{eq} with $u_0, u_1 \in C^\infty _0 (\mathbb{R})$ and  $a \in \mathbb{N}$ . In Dreher's paper \cite{dre}, he also showed the local solvability for $\partial^2 _t u =\partial_x (\vert\partial_x u\vert^{p-2} \partial_x u)$ with $p >5$ and $p \in \mathbb{N}$ and under the initial condition that $u_0, u_1 \in C^k _0 (\mathbb{R}^n)$ for a large natural number $k$. Due to the use of an estimate with loss of regularity (see \eqref{re-loss}), sufficient regularity assumptions need to be made on the initial values and coefficients. Their proof is based on the Nash-Moser implicit function theorem and the argument in Oleinik's paper \cite{ol}. As mentioned below, the standard energy inequality without a loss of the regularity does not work for the quasilinear wave equations with the time or spatial degeneracy. It also should be mentioned that \eqref{eq} can be solved by the Cauchy-Kovalevskaya theorem if $a\in \mathbb{N}$  and initial data are analytic.
On the other hand, there have also been constructions of solutions exhibiting transitions between degenerate and non-degenerate regimes (see  Hu and Wang \cite{HW}, Sugiyama \cite{s3} and Speck \cite{Speck2017}).
Next, we review results that establish the adverse effects of degeneracy on the regularity and well-posedness of solutions.  
There are several works proving that the qualitative properties of solutions change as the type of the equation varies.  
It has been proved by Lerner et al.\ that when an equation loses hyperbolicity and acquires ellipticity, the continuous dependence of solutions on the initial data breaks down, that is, the problem becomes ill-posed (\cite{Lerner2004} and \cite{LernerNguyenTexier2018}).
In this paper we deal with the breakdown of strict hyperbolicity, which is a more naive transition compared to the loss of hyperbolicity into ellipticity. 
Thus, such instabilities are not expected. 
However, we are able to prove a transition in the nature of solutions in which a loss of regularity occurs at the degenerate point.
The degeneracy of the equation is closely related to physical phenomena.  
For instance, in the case of wave equations in elasticity, degeneracy corresponds to fracture or yielding of the material.  
Our study also seems to be related to the study of the compressible Euler equations with physical vacuum. In Jang and  Masmoudi \cite{JM} and  Liu and Yang \cite{LY}, they have studied the following 1D damped Euler equations:
\begin{eqnarray*}\left\{
\begin{array}{ll}
\rho_t + (\rho u)_x =0, \\
\rho u_t + \rho u u_x + p(\rho)_x = - \rho u,
\end{array}\right.
\end{eqnarray*}
where $p(\rho)= \rho^\kappa$ with $\kappa >1$. 
When $\rho$ has zeros, this can also be regarded as a degeneracy of the equation.  
Such degeneracy corresponds to the vacuum state of a fluid.
The degeneracy of the equation is also related to the incompressible Navier--Stokes system. 
When the problem is posed in the half-space, a degenerating viscosity coefficient gives rise to a boundary layer. 
For the Prandtl equation governing this layer, it is known that—analogously to (weakly) hyperbolic equations—the Cauchy problem is well posed in analytic function classes, whereas it is ill posed in Sobolev spaces (see G\'erard-Varet and Dormy\cite{GerardVaretDormy2010}).

\subsection*{Idea of proof}
Our approach is based on generalized d'Alembert-type formulae derived
along degenerate characteristic curves. The main point is to show that,
when \(\beta<\alpha\), the leading order of the solution near the
degenerate point is governed by \(u_1\), not by \(u_0\).

For the linear equation with time-dependent coefficient \(c(t,x)\), the
generalized d'Alembert formula contains additional remainder terms
coming from \(c_t u\). To control these terms, we first prove a refined
upper estimate. More precisely, we introduce a time-dependent
regularized weight adapted to
\[
  x^\alpha+t x^\beta .
\]
This yields
\[
  |u(t,x)|\leq C(x^\alpha+t x^\beta ),
\]
and in particular
\[
  |u(t,x)|\leq Cx^\beta
\]
near \(x=0\). This refined estimate is used to absorb the remainder
terms in the lower bound argument.

For the lower bound, we use the positivity of \(u_1\) in the
generalized d'Alembert formula. In the linear case, the main positive
term is of order \(t x^{p+\beta}\), while the remainder terms are
controlled by
\[
  C t x^{p+\alpha}+C t^2 x^{p+\beta}.
\]
These terms can be absorbed by taking the spatial neighborhood and the
time interval sufficiently small. In the quasilinear case, the argument
is based on a successive iteration: starting from the lower estimate
with the order of \(u_0\), the iteration improves the order step by step
and yields, in the limit,
\[
  u(t,x)\geq Ctx^\beta .
\]

Together with the upper bound, this shows that the leading order changes
from \(x^\alpha\) to \(x^\beta\), which proves the loss of regularity at
the degenerate point.

\subsection*{Notation}
Throughout this paper, the constant \(C\) may change from line to line.
When it is necessary to emphasize the dependence of constants, we add
subscripts such as \(C_\delta\) or \(C_T\).  Constants denoted by
\(C_\delta\) may depend on \(\delta\), and constants denoted by \(C_T\)
may depend on \(T\). 

For a space-time domain \(D_T\), we denote the \(L^\infty\)-norm on
\(D_T\) by \(\|\cdot\|_{L^\infty(D_T)}\).  When
\(D_T=[0,T]\times[0,1]\), we also write
\[
  \|\cdot\|_{L^\infty_TL^\infty}
  =
  \|\cdot\|_{L^\infty([0,T]\times[0,1])}.
\]

\section{Proof of Theorem \ref{thl}}
Throughout the proofs of Theorems \ref{thl}, \ref{thq}, and
\ref{thq2}, we suppose that \(0<\delta_0\leq \delta\leq 1/2\).

Throughout Sections 2--4, for \(\zeta\in\mathbb R\) and
\(\varepsilon>0\), we use the notation
\[
  w_\zeta(x)=x^{-\zeta},
  \qquad
  w_{\zeta,\varepsilon}(x)=(x+\varepsilon)^{-\zeta}.
\]

To derive a generalized d'Alembert-type formula formally, we temporarily
impose stronger regularity conditions such that
\(u\in C^1([0,T]\times[0,\delta])\). This derivation will be justified
later. Integrating the equation \eqref{vdwv}, we have
\[
\partial_t\int_0^x u_t(t,\xi)\,d\xi
=
c(t,x)^2u_x(t,x)-c(t,0)^2u_x(t,0).
\]
Since \(c(t,0)=0\), this gives
\[
\partial_t\int_0^x u_t(t,\xi)\,d\xi
=
c(t,x)^2u_x(t,x).
\]

From this identity, setting
\[
\left\{
\begin{aligned}
w(t,x) &= \int_0^x u_t(t,\xi)\,d\xi + c(t,x)u(t,x),\\
z(t,x) &= \int_0^x u_t(t,\xi)\,d\xi - c(t,x)u(t,x),
\end{aligned}
\right.
\]
we obtain
\[
\left\{
\begin{aligned}
w_t - c(t,x)w_x
&=
\bigl(c_t(t,x)-c(t,x)c_x(t,x)\bigr)u,\\
z_t + c(t,x)z_x
&=
-\bigl(c_t(t,x)+c(t,x)c_x(t,x)\bigr)u.
\end{aligned}
\right.
\]

We denote by \(x_\pm(s)=x_\pm(s;t,x)\) the characteristic curves passing
through the point \((t,x)\) at time \(s=t\). These curves are defined as
the solutions to the initial value problem
\[
\frac{d}{ds}x_\pm(s)=\pm c(s,x_\pm(s)),
\qquad
x_\pm(t)=x.
\]
That is, \(x_\pm(s;t,x)\) satisfies
\begin{align} \label{int-c}
x_\pm(s)
=
x\pm \int_t^s c(\sigma,x_\pm(\sigma))\,d\sigma .
\end{align}
For the characteristic curves, we consider them only within the spatial
domain \([0,\delta]\). Characteristic curves that leave this interval
are not taken into account.

Along the characteristic curves \(x_\pm(s)=x_\pm(s;t,x)\), we have
\[
\frac{d}{ds}w(s,x_-(s))
=
\bigl(c_t-cc_x\bigr)(s,x_-(s))u(s,x_-(s)),
\]
and
\[
\frac{d}{ds}z(s,x_+(s))
=
-\bigl(c_t+cc_x\bigr)(s,x_+(s))u(s,x_+(s)).
\]
We set
\[
V_-(s)=\bigl(c_t-cc_x\bigr)(s,x_-(s)),
\qquad
V_+(s)=\bigl(c_t+cc_x\bigr)(s,x_+(s)).
\]

Thus, integrating from \(s=0\) to \(s=t\), we obtain
\begin{align}
w(t,x)
&=
w(0,x_-(0))
+
\int_0^t V_-(s)u(s,x_-(s))\,ds, \label{wfm}\\
z(t,x)
&=
z(0,x_+(0))
-
\int_0^t V_+(s)u(s,x_+(s))\,ds. \label{zfm}
\end{align}
Subtracting these two identities gives
\begin{align}
2c(t,x)u(t,x)
=&\, w(t,x)-z(t,x) \notag\\
=&\, w(0,x_-(0))-z(0,x_+(0)) \notag\\
&+
\int_0^t
\left\{
V_-(s)u(s,x_-(s))
+
V_+(s)u(s,x_+(s))
\right\}\,ds.
\label{wzV}
\end{align}

\begin{Prop}\label{just}
Under the conditions \eqref{27}, \eqref{28}, the formulas
\eqref{wfm}, \eqref{zfm}, and \eqref{wzV} hold.
\end{Prop}

\begin{proof}
We only justify \eqref{wfm}. For arbitrary \(\varepsilon>0\), we set,
for \(x\in(0,\delta]\),
\[
w_\varepsilon(t,x)
=
\int_\varepsilon^x u_t(t,\xi)\,d\xi
+
c(t,x)u(t,x).
\]
From \eqref{vdwv}, we have
\[
\begin{aligned}
(w_\varepsilon)_t-c(t,x)(w_\varepsilon)_x
&=
\bigl(c_t(t,x)-c(t,x)c_x(t,x)\bigr)u(t,x)  \\
&\quad
-c(t,\varepsilon)^2u_x(t,\varepsilon).
\end{aligned}
\]
Thus, along the characteristic curve \(x_-(s)=x_-(s;t,x)\), we obtain
\[
\begin{aligned}
w_\varepsilon(t,x)
=&\,
w_\varepsilon(0,x_-(0))
+
\int_0^t
V_-(s)u(s,x_-(s))\,ds \\
&-
\int_0^t
c(s,\varepsilon)^2u_x(s,\varepsilon)\,ds .
\end{aligned}
\]
Since \(c(t,0)=0\) and \(c(t,x)u_x(t,x)\in C([0,T]\times[0,\delta])\),
we have
\[
c(s,\varepsilon)^2u_x(s,\varepsilon)
=
c(s,\varepsilon)\{c(s,\varepsilon)u_x(s,\varepsilon)\}
\longrightarrow 0
\]
uniformly in \(s\in[0,T]\) as \(\varepsilon\to0\). Moreover,
\(w_\varepsilon(0,x_-(0))\to w(0,x_-(0))\). Hence, taking
\(\varepsilon\to0\), we obtain \eqref{wfm}. The formula \eqref{zfm} is
justified in the same way, and \eqref{wzV} follows by subtracting
\eqref{zfm} from \eqref{wfm}.
\end{proof}

\begin{lemma}\label{lem:char-bounds}
Fix $p \geq 1$. For any $t,s \in [0,T]$ and $x \in (0,\delta]$ such that $x_\pm (s;t,x) \in [0, \delta]$, we have
\begin{align*}
C^{-1} _T  x\leq x_\pm (s;t,x) \leq C_T x,
\end{align*}
here the constant $C_T$ is such that $C_T \to 1$ as $T \to 0$.
In particular, it follows that for  $\zeta  \in \R  $
\begin{align*}
C^{-1} _T \wz(x) \leq \wz(x_\pm (s)) \leq C_T \wz(x).
\end{align*}
\end{lemma}

\begin{proof}
    Since this lemma can be proved in the same manner as Lemma~3.1 below, we omit the proof here.
\end{proof}

From the second inequality in Lemma \ref{lem:char-bounds}, we also have that
\begin{align*}
C^{-1} _T  \wze(x) \leq \wze (x_\pm (s)) \leq C_T \wze (x). 
\end{align*}
If $\zeta \geq 0$, it holds that
\begin{align*}
\wze (x) \leq \wz (x).
\end{align*}

We will first show the  upper estimate for sufficiently small
\(T>0\):
\begin{align} \label{d-in}
|u(t,x)| \leq C_\delta\bigl(x^\alpha+t x^\beta\bigr),
\end{align}
where \(C_\delta\) is a positive constant depending on \(\delta\).
In particular, since \(\alpha>\beta\), this implies
\[
|u(t,x)|\leq C_\delta x^\beta
\]
on \([0,T]\times[0,\delta_0]\).

From \eqref{wzV}, we have
\begin{align}
2c(t,x)u(t,x)
=&\; \int_{x_+(0)}^{x_-(0)} u_1(y)\,dy \notag\\
&+ c(0,x_-(0))u_0(x_-(0))
 + c(0,x_+(0))u_0(x_+(0)) \notag\\
&+ \int_0^t
\left\{
V_-(s)u(s,x_-(s))
+
V_+(s)u(s,x_+(s))
\right\}\,ds .
\label{u-ab1}
\end{align}
By the mean value theorem, Lemma~\ref{lem:char-bounds}, and
\eqref{int-c}, the first term of \eqref{u-ab1} is estimated, for some
\(\xi\in(x_+(0),x_-(0))\), as
\begin{align}
\int_{x_+(0)}^{x_-(0)} u_1(y)\,dy
=&\; u_1(\xi)\bigl(x_-(0)-x_+(0)\bigr) \notag\\
\leq&\; Cx^\beta
\int_0^t
c(s,x_-(s))+c(s,x_+(s))\,ds \notag\\
\leq&\; C t x^{p+\beta}. \notag
\end{align}
From Lemma~\ref{lem:char-bounds} and the assumption
\(\alpha>\beta\), the second and third terms of \eqref{u-ab1} are
estimated by
\begin{align}
&c(0,x_-(0))u_0(x_-(0))
+
c(0,x_+(0))u_0(x_+(0)) \notag\\
&\qquad\leq Cx^p x^\alpha .
\notag
\end{align}
Moreover, by the assumptions on \(c_t\), \(c\), and \(c_x\), we have
\[
|V_\pm(s)|
\leq Cx^p
\]
along the characteristic curves. Therefore,
\begin{align*}
&\left|
\int_0^t
\left\{
V_-(s)u(s,x_-(s))
+
V_+(s)u(s,x_+(s))
\right\}\,ds
\right| \\
&\qquad\leq
Cx^p\int_0^t
\left(
|u(s,x_-(s))|+|u(s,x_+(s))|
\right)\,ds .
\end{align*}

We now introduce the regularized weight adapted to the estimate
\eqref{d-in}. Put
\[
\rho_\varepsilon(t,x)
=
(x+\varepsilon)^\alpha+t(x+\varepsilon)^\beta,
\qquad
W_\varepsilon(t,x)=\rho_\varepsilon(t,x)^{-1}.
\]
We multiply \eqref{u-ab1} by \(W_\varepsilon(t,x)/c(t,x)\) and take
the supremum over \([0,T]\times(0,\delta/2]\).

The terms involving \(u_1\) and \(u_0\) are bounded uniformly in
\(\varepsilon\). Indeed, using \(c(t,x)\geq L_0x^p\), we have
\[
\frac{W_\varepsilon(t,x)}{c(t,x)}\,t x^{p+\beta}
\leq
C\frac{t x^\beta}{(x+\varepsilon)^\alpha+t(x+\varepsilon)^\beta}
\leq C,
\]
and
\[
\frac{W_\varepsilon(t,x)}{c(t,x)}\,x^{p+\alpha}
\leq
C\frac{x^\alpha}{(x+\varepsilon)^\alpha+t(x+\varepsilon)^\beta}
\leq C.
\]

It remains to estimate the integral term in \eqref{u-ab1}. For the
part of the integral for which \(x_\pm(s;t,x)\leq \delta/2\), Lemma
\ref{lem:char-bounds} gives
\[
\rho_\varepsilon(s,x_\pm(s;t,x))
\leq
C_T\rho_\varepsilon(t,x).
\]
Hence
\[
\begin{aligned}
|u(s,x_\pm(s;t,x))|
&\leq
\left(
\sup_{[0,T]\times[0,\delta/2]}
W_\varepsilon(t,x)|u(t,x)|
\right)
\rho_\varepsilon(s,x_\pm(s;t,x))  \\
&\leq
C_T
\left(
\sup_{[0,T]\times[0,\delta/2]}
W_\varepsilon(t,x)|u(t,x)|
\right)
\rho_\varepsilon(t,x).
\end{aligned}
\]
Since \(|V_\pm(s)|\leq Cx^p\) along the characteristic curves, this
part is bounded, after multiplication by
\(W_\varepsilon(t,x)/c(t,x)\), by
\[
C\,T
\sup_{[0,T]\times[0,\delta/2]}
W_\varepsilon(t,x)|u(t,x)|.
\]

On the other hand, for the part of the integral for which
\(x_\pm(s;t,x)\geq \delta/2\), Lemma~\ref{lem:char-bounds} implies that
\(x\geq C_T^{-1}\delta/2\). Thus \(W_\varepsilon(t,x)\),
\(1/c(t,x)\), and \(V_\pm(s)\) are bounded by constants depending on
\(\delta\). Since \(u\) is continuous on \([0,T]\times[0,\delta]\), this
part is bounded by $C_\delta T.$
Combining these estimates, we obtain
\[
\sup_{[0,T]\times[0,\delta/2]}
W_\varepsilon(t,x)|u(t,x)|
\leq
C_1
+
C_2T
\sup_{[0,T]\times[0,\delta/2]}
W_\varepsilon(t,x)|u(t,x)|
+
C_\delta T.
\]
Thus, for sufficiently small \(T>0\),
\[
\sup_{[0,T]\times[0,\delta/2]}
W_\varepsilon(t,x)|u(t,x)|
\leq C.
\]
Taking \(\varepsilon\to0\), we obtain
\[
|u(t,x)|
\leq
C_\delta\bigl(x^\alpha+t x^\beta\bigr).
\]
This proves \eqref{d-in}.

Next we will show the lower estimate on \([0,T]\times[0,\delta_0]\)
\begin{align} \label{d-inl}
u(t,x)\geq Ctx^\beta
\end{align}
for sufficiently small \(T\) and \(\delta_0\).
Applying \eqref{d-in} to \eqref{u-ab1}, we have, for sufficiently small
\(\delta_0>0\) and \(T>0\),
\begin{align}
2c(t,x)u(t,x)
\geq&
\int_{x_+(0)}^{x_-(0)}u_1(y)\,dy
\notag\\
&-
\left|
\int_0^t
\left\{
V_-(s)u(s,x_-(s))
+
V_+(s)u(s,x_+(s))
\right\}\,ds
\right|.
\notag
\end{align}
Here we have used the non-negativity of the terms involving \(u_0\).

By the mean value theorem, Lemma~\ref{lem:char-bounds}, and
\eqref{int-c}, the first term is estimated from below as
\begin{align}
\int_{x_+(0)}^{x_-(0)}u_1(y)\,dy
=&\;
u_1(\xi)\bigl(x_-(0)-x_+(0)\bigr)
\notag\\
=&\;
u_1(\xi)
\int_0^t
c(s,x_-(s))+c(s,x_+(s))\,ds
\notag\\
\geq&\;
C x^\beta\int_0^t x^p\,ds
\notag\\
\geq&\;
C t x^{p+\beta},
\notag
\end{align}
where \(\xi\in(x_+(0),x_-(0))\).

On the other hand, from the assumptions on \(c_t\), \(c\), and \(c_x\),
we have, along the characteristic curves,
\[
 |V_\pm(s)|\leq Cx^p.
\]
Hence, by \eqref{d-in} and Lemma~\ref{lem:char-bounds},
\begin{align}
&\left|
\int_0^t
\left\{
V_-(s)u(s,x_-(s))
+
V_+(s)u(s,x_+(s))
\right\}\,ds
\right|
\notag\\
&\qquad\leq
C x^p
\int_0^t
\left(
|u(s,x_-(s))|+|u(s,x_+(s))|
\right)\,ds
\notag\\
&\qquad\leq
C x^p
\int_0^t
\left(x^\alpha+s x^\beta\right)\,ds
\notag\\
&\qquad\leq
C t x^{p+\alpha}
+
C t^2x^{p+\beta}.
\notag
\end{align}
Since \(\alpha>\beta\), we have
\[
x^{p+\alpha}\leq \delta_0^{\alpha-\beta}x^{p+\beta}
\qquad
\text{for }0<x\leq\delta_0.
\]
Therefore,
\begin{align}
2c(t,x)u(t,x)
\geq&
C_1t x^{p+\beta}
-
C_2t\delta_0^{\alpha-\beta}x^{p+\beta}
-
C_3t^2x^{p+\beta}
\notag\\
\geq&
C_4t x^{p+\beta},
\notag
\end{align}
by taking \(\delta_0>0\) and \(T>0\) sufficiently small. Since
\(c(t,x)\leq L_1x^p\), we obtain
\[
u(t,x)\geq Ctx^\beta.
\]
This proves \eqref{d-inl}.

\section{Proof of Theorem \ref{thq}}
We suppose that $\delta \leq 1/2$.
We set $w$ and $z$ as follows
\begin{eqnarray}\label{wz}\left\{
\begin{array}{ll}\displaystyle
w =\int_0 ^x u_t dx + \frac{u^{a+1}}{a+1}, \\
\displaystyle z =\int_0 ^x u_t dx - \frac{u^{a+1}}{a+1}.
\end{array}\right.
\end{eqnarray}
Then we have \eqref{eq}
\begin{align} \label{wz-eq}
\begin{cases}
w_t - u^a w_x = 0, \\
z_t + u^a z_x = 0.
\end{cases}
\end{align}

Let $x_{-} (s)$ and $x_{+} (s)$ be characteristic curves on the first and second equations of \eqref{wz-eq}
respectively. That is,  $x_{+} (s)$ and $x_{-} (s)$ are solutions to the following differential equations respectively:
 \begin{eqnarray}\label{cc}
\dfrac{d}{ds} x_{\pm} (s)=\pm u^a(s,x_{\pm} (s)).
\end{eqnarray}
When we emphasize that the characteristic curves go through $(t,x)$, we denote $x_{\pm} (s)$ by  $x_{\pm} (s; t,x)$.
That is, $x_{\pm} (s; t,x)$ satisfies that
\begin{eqnarray} \label{int-x}
x_{\pm} (s; t,x) = x \pm \int_t ^s  u^a (\tau, x_{\pm} (\tau; t,x)) d\tau.
\end{eqnarray}
On the characteristic curves, we have
\begin{align} \label{wz-qfm}
\begin{cases}
w(t, x) = w(0, x_{-}(0)), \\
z(t, x) = z(0, x_{+}(0)).
\end{cases}
\end{align}
For the characteristic curves, we consider them only within the spatial domain $[0,\delta]$ of $u$.  We note that the characteristic curves exist uniquely, since $u(t,x)^a$ is $C^1$ function on $[0,T] \times [0, \delta]$.
By choosing \( x_{\pm}(s) = x_{\pm}(s, t, x) \), we have the following generalized d'Alembert's formula:
\begin{align}
\frac{2 u^{a + 1}(t,x)}{a + 1}
&= w(t, x) - z(t, x) \notag \\
&= w(0, x_{-}(0)) - z(0, x_{+}(0)) \notag \\
&= \int_{x_{+}(0)}^{x_{-}(0)} u_1(\xi) \, d\xi
+ \frac{1}{a+ 1} \left( u_0^{a + 1}(x_{+}(0)) + u_0^{a + 1}(x_{-}(0)) \right). \label{we-df}
\end{align}
The derivations of \eqref{wz-qfm} and \eqref{we-df} can be justified in the same way as in the proof of Proposition \ref{just}. The next lemma plays a fundamental role in the derivation of weighted estimates.
This type of lemma has  already  been shown  in \cite{YHYS1}.  However, we give a simpler proof here.
\begin{lemma} \label{es-cq}
Let $\zeta \in \R$ and $a\geq 1$.
Assume that
\[
u\in C([0,T]\times[0,\delta]),\qquad
u(t,x)\geq 0,\qquad u(t,0)=0,
\]
and
\[
u^{a-1}u_x\in C([0,T]\times[0,\delta]).
\]
For $t,s \in [0,T]$ and $x\in(0,\delta]$, let
$x_\pm(s;t,x)$ be the characteristic curves defined by \eqref{cc}.
Then the following estimate holds as long as $x_\pm(\cdot;t,x)$ remains
within $[0,\delta]$:
\begin{eqnarray} \label{w-x-es}
C^{-1}_T x \leq x_{\pm}(s;t,x) \leq C_T x.
\end{eqnarray}
Here the constant $C_T$ satisfies $C_T\to1$ as $T\to0$.
In particular, the following estimates hold:
\begin{align}
C^{-1}_T \wze(x)
\leq \wze(x_\pm(s;t,x))
\leq C_T \wze(x), \label{1} \\
C^{-1}_T \wz(x)
\leq \wz(x_\pm(s;t,x))
\leq C_T \wz(x). \label{2}
\end{align}
\end{lemma}

\begin{proof}
We only prove the estimate for \(x_+(s;t,x)\). The estimate for
\(x_-(s;t,x)\) is obtained in the same way.

First we note that \(u^a\) is Lipschitz continuous with respect to
\(x\), uniformly in \(t\). Indeed, on \((0,\delta]\),
\[
\partial_x(u^a)=a u^{a-1}u_x,
\]
and the right-hand side is bounded on \([0,T]\times[0,\delta]\).
Hence the characteristic equation \eqref{cc} has a unique solution.
Moreover, since \(u(t,0)=0\), \(x\equiv0\) is also a solution of
\eqref{cc}. By uniqueness, a characteristic starting from \(x>0\)
cannot meet \(x=0\). Therefore
\[
x_+(s;t,x)>0
\]
as long as the characteristic remains in \([0,\delta]\).

From \eqref{cc}, we have
\[
\frac{d}{ds}x_+(s;t,x)
=
u(s,x_+(s;t,x))^a.
\]
Thus
\[
\frac{d}{ds}\log x_+(s;t,x)
=
\frac{u(s,x_+(s;t,x))^a}{x_+(s;t,x)}.
\]
Since \(u(s,0)=0\), the mean value theorem gives
\[
\frac{u(s,x_+(s;t,x))^a}{x_+(s;t,x)}
=
a u(s,\xi)^{a-1}u_x(s,\xi)
\]
for some \(\xi\in(0,x_+(s;t,x))\). Hence
\[
\left|
\frac{d}{ds}\log x_+(s;t,x)
\right|
\leq
a\|u^{a-1}u_x\|_{L^\infty([0,T]\times[0,\delta])}.
\]
Integrating this inequality between \(t\) and \(s\), we obtain
\[
e^{-MT}x
\leq
x_+(s;t,x)
\leq
e^{MT}x,
\]
where
\[
M=a\|u^{a-1}u_x\|_{L^\infty([0,T]\times[0,\delta])}.
\]
The same argument applies to \(x_-(s;t,x)\), and therefore
\eqref{w-x-es} follows with \(C_T=e^{MT}\).

Finally, \eqref{1} and \eqref{2} follow from \eqref{w-x-es} and the
definitions of \(\wze\) and \(\wz\). Since \(C_T\to1\) as
\(T\to0\), the proof is complete.
\end{proof}

First we prove the following  upper estimate:
\begin{Prop}\label{33}
Under the same assumption of Theorem \ref{thq}, it follows that
\begin{align} \label{d-up}
u(t,x) \leq C x^{\beta} .
\end{align}
\end{Prop}
\begin{proof}
Before proceeding to the proof, we remark that in the argument of this proposition, the non-negativity of $u$ is used only to define $u^a$ and $u^{a+1}$.
Hence, if $a \in \N$, the non-negativity assumption is not required, and instead of the estimate stated in the proposition, the following estimate holds.
\begin{align} \label{3}
|u(t,x)| \leq C x^{\beta} .
\end{align}
Hereafter, we take $T>0$ sufficiently small so that if $x \in [0,\delta/2]$, then $x_\pm(s;t,x) \in [0,\delta]$ for all $s \in [0,t].$
In the same way as in the proof of Theorem \ref{thl}, applying Lemma \ref{es-cq}, the right hand side of \eqref{we-df} is estimated as follows
\begin{align*}
\int_{x_{+}(0)}^{x_{-}(0)} u_1(\xi) \, d\xi
&+ \frac{1}{a+ 1} \left( u_0^{a+ 1}(x_{+}(0)) + u_0^{a + 1}(x_{-}(0)) \right)  \\
\leq & C_1x^{\beta}  \int_0 ^t u(s,x_- (s))^a  + u(s,x_+ (s))^a ds +C_2 x^{\al(a+1)}.
\end{align*}
In the following, we proceed in the same manner as in Theorem \ref{thl}.
That is, we multiply both sides by $\wbe(x)^{a+1}$ and then take the supremum norm, restricting to $x \in [0,\delta/2]$.
For the region $x \geq \delta/2$, we instead employ the boundedness of the weight depending on $\delta$, and estimate the term $\int_{0}^{t} u(s,x_\pm(s))^a  ds$.
By multiplying by $\wbe (x)^{a+1}$, from \eqref{1} and the assumption $\al \geq \bt $ and the fact that $w_{\bt,\varepsilon} (x) \leq w_{\bt} (x)$, we have that
\begin{align*}
(\wbe (x) u(t,x))^{a+1} \leq & C_1  T \sup_{[0,T]\times[0,\delta/2]} (\wbe (x) u(t,x))^{a} \notag \\
& +  CT\sup_{[0,T]\times[\delta/2,\delta]}(\wbe (x) u(t,x))^{a}  + C_2,
\end{align*}
from which, we have that
\begin{align*}
 (\sup_{[0,T]\times[0,\delta/2]} \wbe (x) u(t,x))^{a} ( \sup_{[0,T]\times[0,\delta/2]} \wbe (x) u(t,x) -C_1 T) \leq C_\delta T +  C_2 .
\end{align*}
Thus we have the uniform boundedness of $ \wbe (x) u(t,x)$ with $\varepsilon$.
Taking $\varepsilon \rightarrow 0$, we obtain the desired estimate \eqref{d-up}.
\end{proof}

Next we show the lower estimate:
\begin{Prop}\label{34}
Under the same assumption of Theorem \ref{thq}, it follows that for sufficiently small $T>0$
\begin{align} \label{l-d}
    u(t,x) \geq C t x^{\beta} .
\end{align}
\end{Prop}
\begin{proof}
In this proof, we  also take $T>0$ sufficiently small so that if $x \in [0,\delta/2]$, then $x_\pm(s;t,x) \in [0,\delta]$ for all $s \in [0,t].$
First we note from the non-negativity of \(u(t,x)\) that
\begin{align} \label{ch-re}
x_-(s;t,x)\geq x\geq x_+(s;t,x),
\qquad 0\leq s\leq t .
\end{align}
Indeed, by \eqref{int-x}, for \(0\leq s\leq t\),
\[
x_-(s;t,x)
=
x+\int_s^t u^a(\tau,x_-(\tau;t,x))\,d\tau
\geq x,
\]
and
\[
x_+(s;t,x)
=
x-\int_s^t u^a(\tau,x_+(\tau;t,x))\,d\tau
\leq x.
\]
 From the choice of initial data, applying Lemma \ref{es-cq} to \eqref{we-df}, we have
\begin{align*} 
  \frac{2u^{a+1} (t,x)}{a+1} \geq  Cx^{\al(a+1)} ,
\end{align*}
which implies that
\begin{align}
    u(t,x) \geq C_1 x^{\al}. \label{1st}
\end{align}
In the same way as in the upper estimate, \eqref{we-df} yields that
\begin{align}
u(t,x)^{a+1} \geq & C_0 x^{\beta} \int_0 ^t (u^a (s, x_+ (s)) + u^a (s,x_- (s)))ds \notag \\
\geq & C_0 x^{\beta}  \int_0 ^t u^a (s, x_+ (s))ds, \label{pr}
\end{align}
here the non-negativity of $u$ is used.
By \eqref{ch-re}, the trajectory $x_{-}(s)$ may extend beyond the interval $[0,\delta/2]$, but this no longer needs to be taken into account.
Applying \eqref{1st} estimate to \eqref{pr}, we have that on [0,T]
\begin{align*}
u(t,x)^{a+1} \geq C_0   C^a_1 C^a _2  x^{\bt +a\alpha}t ,
\end{align*}
where the constant $C_2$ appearing here is the same as that in Lemma \ref{es-cq}.
By repeating this procedure, we obtain a lower bound for $u$.
Suppose that, after $n$ iterations, the following inequality is obtained.
\begin{align*}
u(t,x) \geq K_n t^{d_n} x^{c _n}, 
\end{align*}
where $K_n, d_n , c_n$ are positive constants.
Substituting this estimate into \eqref{pr}, we obtain that
\begin{align*}
u^{a+1}(t,x) 
&\geq C_0 x^{\beta} \cdot K_n^a C_2 ^{a c_n}  x^{c_n a}
   \int_0^t s^{a d_n} \, ds \\
&= C_0 x^{\beta + c_n a} \cdot K_n^a C_2 ^{ac_n} 
   \frac{1}{1 + a d_n} t^{1 + a d_n}.
\end{align*}
Thus we have
\[
u(t,x)\ge 
\left(\frac{1}{1+a d_n}\right)^{\frac{1}{a+1}}
C_0^{\frac{1}{a+1}}
C_2^{\theta c_n}
K_n^{\theta}
t^{\frac{1}{a+1}+\theta d_n}
x^{\frac{\beta}{a+1}+\theta c_n}.
\]
where we put $\theta =a/(1+a)$. Hence, the following recurrence relations are derived.
\begin{align*}
d_{n+1}  &= \frac{1}{a+1}+\theta d_n,\\[2mm]
c_{n+1}&= \frac{\bt}{a+1} + \theta c_n,
\end{align*}
and
\[
K_{n+1}
= \left(\frac{1}{1+a d_n}\right)^{\frac{1}{a+1}}
  C_0^{\frac{1}{a+1}}\, C_2^{\theta c_n}\, K_n^{\theta}
= \left(\frac{C_0\, C_2^{ac_n}}{1+a d_n}\right)^{\frac{1}{a+1}} K_n^{\theta}.
\]
The first and second recurrence relations can be solved easily, and we obtain the following.
\begin{align*}
d_n      &= 1 + \theta^{\,n}\,(d_0-1),\\[2mm]
c_n      &= \beta + \theta^{\,n}\,(c_0-\beta),
\end{align*}
where $d_0 =0$, $c _0= \alpha$. From the boundedness and the positivity of $d_n$ and $c_n$, we see that there exists a positive constant $C_3$ such that
\begin{align*}
K_{n+1} \geq C_3 K_n ^\theta.
\end{align*}
Solving this recurrence inequality, we find that
\[
K_n \ge K_0^{\theta^{\,n}}\, C_3^{\frac{1-\theta^{\,n}}{1-\theta}} ,
\]
where $K_0 =C_1$.
Therefore, we have that
\begin{align*}
u(t,x) \geq  C_1 ^{\theta^{\,n}}\, C_3^{\frac{1-\theta^{\,n}}{1-\theta}} t^{1 - \theta^{\,n}\,} x^{\beta + \theta^{\,n}\,(\alpha-\beta)}.
\end{align*}
Thus, taking $n \rightarrow \infty$, we have \eqref{l-d} from $\theta^n \rightarrow 0$.
\end{proof}
From Propositions \ref{33} and \ref{34}, the proof of Theorem \ref{thq} is completed.

\begin{remark}[On the positivity of $u(t,x)$]
As pointed out in Remark \ref{re1}, the non-negativity of $u(t,x)$ is not required if $u^m u_{xx}$ is bounded for some natural number $m$ such that $m <2a$ and if $a$ is a natural number.  
In fact, integrating  \eqref{eq} on $[0,t]$ twice, we have
\begin{align*}
 u(t,x) = u_0(x) + t u_1 (x) + \int_0^t \int_0 ^s u^{2a} u_{xx} (\tau,x) + 2a u^{2a-1} u^2_x (\tau,x)d\tau ds. 
\end{align*}
We note that the assumption that $a \in \N$ is used to define $u^a$.
Applying the boundedness of $u^m u_{xx}$ and $u^{a-1} u_x$ together with inequality \eqref{3} (which holds even without the non-negativity assumption), we obtain
\begin{align*}
    u(t,x) \geq K_0 x^{\alpha} + t K_2 x^{\beta} - C t^2 ( x^{\beta(2a-m)} + x^{\bt}).
\end{align*}
Since $a$ is a natural number, we have $2a-m \geq 1$, and hence, for sufficiently small $T>0$,
\[
    u(t,x) \geq C t x^{\beta}.
\]
In other words, not only the non-negativity but also the lower bound estimate stated in Theorem \ref{thq} can be derived.  
\end{remark}

\section{Proof of Theorem \ref{thq2}}
We prove Theorem~\ref{thq2}.
Given the Riemann invariants \eqref{wz}, the solution to \eqref{veq} formally satisfies the following hyperbolic system:
\begin{eqnarray}\label{wzv}\left\{
\begin{array}{ll}\displaystyle
w_t - u^{a} w_x = -\,a\int_{0}^{x} u^{\,2a-1}(t,\xi)\,\bigl(u_x(t,\xi)\bigr)^2\,d\xi, \\\displaystyle
z_t + u^{a} z_x = -\,a\int_{0}^{x} u^{\,2a-1}(t,\xi)\,\bigl(u_x(t,\xi)\bigr)^2\,d\xi.
\end{array}\right.
\end{eqnarray}
The derivations of \eqref{wzv} and \eqref{we-df2} can be justified
in the same way as in the proof of Proposition~\ref{just}, by first
integrating over \([\varepsilon,x]\) and then letting
\(\varepsilon\downarrow0\).

In the same way as in the derivation of \eqref{we-df}, we have that
\begin{align}
\frac{2 u^{a + 1}(t,x)}{a + 1}
=& \int_{x_{+}(0)}^{x_{-}(0)} u_1(\xi) \, d\xi
+ \frac{1}{a + 1} \left( u_0^{a + 1}(x_{+}(0)) + u_0^{a + 1}(x_{-}(0)) \right) \notag \\
&- R(t,x),
\label{we-df2}
\end{align}
where we put the positive remainder term 
\begin{align*}
R(t,x) =   a \int_0 ^t \int_{x_+ (\tau)} ^{x_- (\tau)} u^{\,2a-1}(\tau,\xi)\,\bigl(u_x(\tau,\xi)\bigr)^2\,d\xi d\tau.
\end{align*}
Since $2a-1 \geq 1$ from $a \geq 1$, from the boundedness of $u^{a-1}u_x, u$ and the mean value theorem, $R$ can be estimated by
\begin{align*}
R(t,x) \leq & C \int_0 ^t \int_{x_+ (\tau)} ^{x_- (\tau)} u(\tau,\xi)\,d\xi  d\tau \\
\leq & C \int_0 ^t (x_- (\tau) - x_+ (\tau))u(\tau, \xi) d\tau \\
=&  C \int_0 ^t \int_\tau ^t u(\tau', x_- (\tau'))^a +  u(\tau', x_+ (\tau'))^a d\tau' u(\tau, \xi) d\tau
\end{align*}
for $\xi \in [x_+ (\tau), x_- (\tau)]$. 
Thus, from the same argument as in the proof of Proposition \ref{33}, we obtain the upper estimate on $x \in [0,\delta/2]$ for sufficiently small $T>0$
\begin{align}\label{44}
u(t,x) \leq Cx^{\bt}.
\end{align}
To derive the lower bound, we estimate the remainder term more precisely.
Put
\[
A(s)=u(s,x_-(s;t,x))^a+u(s,x_+(s;t,x))^a .
\]
By Lemma 3.1 and the lower bound for \(u_1\), we have
\[
\int_{x_+(0)}^{x_-(0)}u_1(y)\,dy
\ge Cx^{\beta}(x_-(0)-x_+(0))
= Cx^{\beta}\int_0^t A(s)\,ds .
\]
On the other hand, since \(u^{a-1}u_x\) is bounded, we have
\[
u^{2a-1}u_x^2
=
u\,(u^{a-1}u_x)^2
\le Cu .
\]
Using the upper estimate \(u(t,x)\le Cx^{\beta}\) and Lemma \ref{es-cq} again, we obtain
\[
\begin{aligned}
R(t,x)
&\le C\int_0^t\int_{x_+(\tau)}^{x_-(\tau)}
u(\tau,\xi)\,d\xi d\tau  \\
&\le Cx^{\beta}\int_0^t
(x_-(\tau)-x_+(\tau))\,d\tau  \\
&= Cx^{\beta}\int_0^t\int_\tau^t A(s)\,ds d\tau  \\
&= Cx^{\beta}\int_0^t sA(s)\,ds  \\
&\le CTx^{\beta}\int_0^t A(s)\,ds .
\end{aligned}
\]
Therefore, taking \(T>0\) smaller if necessary, the term \(R(t,x)\)
can be absorbed into the positive term involving \(u_1\). Hence, from \eqref{we-df2}, we obtain
\[
u(t,x)^{a+1}
\ge
C x^{\alpha(a+1)}
+
C x^{\beta}\int_0^t u(s,x_+(s;t,x))^a\,ds .
\]
Consequently, the estimates corresponding to \eqref{1st} and \eqref{pr}
hold, and the same iteration argument as in Proposition~\ref{34} gives
\[
u(t,x)\ge Ctx^{\beta}.
\]

\section{Comparison with an explicit representation formula}

We end this paper by comparing our result with a model equation for
which an explicit representation formula is available. This comparison
is not used in the proof of the main theorems. Its purpose is to show
that, in special cases where a representation formula is known, the same
loss of regularity can be seen directly from the formula.

In Galstian and Kinoshita~\cite{GalstianKinoshita2016}, the authors
consider the following one-dimensional degenerate linear equation:
\[
  \partial_t^2 v - x^4\partial_x^2 v = 0.
\]
They provide an explicit representation formula of the form
\[
v(t,x)
=
\frac{1}{2}\Bigl[
(1+tx)\,u_0\!\Bigl(\frac{x}{1+tx}\Bigr)
+
(1-tx)\,u_0\!\Bigl(\frac{x}{1-tx}\Bigr)
\Bigr]
+
\frac{x}{2}
\int_{1/x-t}^{1/x+t}
\eta\,u_1\!\Bigl(\frac{1}{\eta}\Bigr)\,d\eta .
\]
Assume that, near \(x=0\),
\[
  u_0(x)\sim x^\alpha,
  \qquad
  u_1(x)\sim x^\beta,
  \qquad
  \alpha>\beta>0.
\]
Then the first term in the above representation satisfies
\[
\frac{1}{2}\Bigl[
(1+tx)\Bigl(\frac{x}{1+tx}\Bigr)^\alpha
+
(1-tx)\Bigl(\frac{x}{1-tx}\Bigr)^\alpha
\Bigr]
=
x^\alpha\bigl(1+O(x)\bigr).
\]
Thus the contribution of \(u_0\) remains of order \(x^\alpha\).

On the other hand, the second term is governed by \(u_1\). By the
change of variable \(1/\eta=\xi\), we obtain
\[
\frac{x}{2}
\int_{1/x-t}^{1/x+t}
\eta u_1\!\left(\frac{1}{\eta}\right)d\eta
=
\frac{x}{2}
\int_{x/(1+tx)}^{x/(1-tx)}
\frac{u_1(\xi)}{\xi^3}\,d\xi .
\]
Using \(u_1(\xi)\sim \xi^\beta\), put
\[
  \ell=\frac{x}{1+tx},
  \qquad
  r=\frac{x}{1-tx}.
\]
Then
\[
\frac{x}{2}
\int_{\ell}^{r}\xi^{\beta-3}\,d\xi
\]
is the leading contribution of the \(u_1\)-term. If \(\beta\ne2\), then
\[
\int_{\ell}^{r}\xi^{\beta-3}\,d\xi
=
\frac{r^{\beta-2}-\ell^{\beta-2}}{\beta-2}
=
2t\,x^{\beta-1}+o(x^{\beta-1}),
\]
and therefore
\[
\frac{x}{2}
\int_{\ell}^{r}\xi^{\beta-3}\,d\xi
=
t x^\beta+o(x^\beta).
\]
If \(\beta=2\), then
\[
\int_{\ell}^{r}\xi^{-1}\,d\xi
=
\log\frac{r}{\ell}
=
\log\frac{1+tx}{1-tx}
=
2tx+o(x),
\]
and hence the same conclusion follows:
\[
\frac{x}{2}
\int_{\ell}^{r}\xi^{-1}\,d\xi
=
t x^2+o(x^2).
\]
Consequently, for sufficiently small positive time,
\[
  C_1 t x^\beta\leq v(t,x)\leq C_2 x^\beta
\]
near \(x=0\). Thus, even in this explicitly solvable model, the leading
order of the solution changes from the order of \(u_0\) to the order of
\(u_1\) when \(\beta<\alpha\).

This example illustrates the same mechanism as the one proved in the
present paper. The advantage of our approach is that it does not rely on
an explicit representation formula. Instead, using generalized
d'Alembert-type formulae, weighted estimates, and iteration arguments,
we can prove the loss of regularity for a wider class of degenerate
linear equations and also for quasilinear equations.

\section*{Acknowledgements}
Y. Hu was partially supported by the National Natural Science Foundation of China (Nos. 12171130 and 12071106). Y. Sugiyama was partially supported by Grants-in-Aid for Scientific Research (C) (No. 23K03169).

\bibliographystyle{amsplain}
\bibliography{bibliography}

\end{document}